\documentclass[a4paper,10pt]{amsart}
\usepackage{latexsym, amssymb,amsmath,amsthm, bbm}
\usepackage[all]{xy}

\def \To{\longrightarrow}

\def \gr{\operatorname{gr}}

\def \CoT{\operatorname{CoT}}
\def \Ext{\operatorname{Ext}}
\def \Vec{\operatorname{Vec}}
\def \comod{\operatorname{comod}}
\def \C{\mathcal{C}}
\def \D{\Delta}
\def \d{\delta}
\def \e{\varepsilon}

\def \A{\mathcal{A}}
\def \S{\mathcal{S}}
\def \K{\mathcal{K}}
\def \F{\mathcal{F}}
\def \Z{\mathbb{Z}}
\def \g{\mathrm{g}}

\def \k{\mathbbm{k}}
\def \1{\mathbf{1}}
\def \Id{\operatorname{Id}}
\def \Rep{\operatorname{Rep}}

\numberwithin{equation}{section}

\newtheorem{theorem}{Theorem}[section]
\newtheorem{lemma}[theorem]{Lemma}
\newtheorem{proposition}[theorem]{Proposition}
\newtheorem{corollary}[theorem]{Corollary}

\newtheorem{example}[theorem]{Example}
\newtheorem{remark}[theorem]{Remark}

\begin{document}

\title{QUIVER BIALGEBRAS AND MONOIDAL CATEGORIES}



\author{Hua-Lin Huang (Jinan)}
\address{School of Mathematics, Shandong University,
Jinan 250100, China} \email{hualin@sdu.edu.cn}

\author{Blas Torrecillas (Almer\'ia)}
\address{Department of Algebra and Analysis, Universidad de Almer\'ia, 04071 Almer\'ia, Spain}
\email{btorreci@ual.es}

\date{}
\maketitle

\begin{abstract}
We study the bialgebra structures on quiver coalgebras and the
monoidal structures on the categories of locally nilpotent and locally finite quiver
representations. It is shown that the path coalgebra of an arbitrary
quiver admits natural bialgebra structures. This endows the category
of locally nilpotent and locally finite representations of an arbitrary quiver with
natural monoidal structures from bialgebras. We also obtain
theorems of Gabriel type for pointed bialgebras and hereditary
finite pointed monoidal categories.
\end{abstract}

\section{Introduction}

This paper is devoted to the study of natural bialgebra structures
on the path coalgebra of an arbitrary quiver and monoidal structures
on the category of its locally nilpotent and locally finite representations. A further
purpose is to establish a quiver setting for general pointed
bialgebras and pointed monoidal categories.

Our original motivation is to extend the Hopf quiver theory
\cite{cz,cr1,cr2,g,gs,mont2,oz} to the setting of generalized Hopf
structures. As bialgebras are a fundamental generalization of Hopf
algebras, we naturally initiate our study from this case. The basic
problem is to determine what kind of quivers can give rise to
bialgebra structures on their associated path algebras or coalgebras.

It turns out that the path coalgebra of an arbitrary quiver admits
natural bialgebra structures, see Theorem 3.2. This seems a bit surprising at
first sight by comparison with the Hopf case given in \cite{cr2},
where Cibils and Rosso showed that the path coalgebra of a quiver
$Q$ admits a Hopf algebra structure if and only if $Q$ is a Hopf
quiver which is very special. Bialgebra structures on
general pointed coalgebras are also considered via quivers thanks to
the Gabriel type theorem of coalgebras, see \cite{chz, cm}. Similar to the Hopf case obtained in \cite{oz}, we give a
Gabriel type theorem for general pointed bialgebras, see Proposition
3.3.

Our another motivation comes from finite monoidal categories which
are natural generalization of finite tensor categories \cite{eo}. To
the knowledge of the authors, not quite much is known for the
construction and classification of finite monoidal categories which
are not tensor categories, i.e., rigid monoidal categories \cite{egno}. By
taking advantage of the well-developed quiver representation theory,
the quiver presentation of a pointed bialgebra $B$ can help us to
investigate the monoidal category of right $B$-comodules.
Accordingly, some classification results of pointed monoidal
categories are obtained, see Proposition 4.1 and Corollary 4.2.

Bialgebra structures on the path coalgebra of a quiver $Q$ induce
monoidal structures on the category $\Rep^{lnlf} Q$ of locally
nilpotent and locally finite representations of $Q.$ These monoidal structures are
also expected to be useful for the studying of the category
$\Rep^{lnlf} Q$ itself. For example, the tensor product of quiver
representations naturally leads to the Clebsch-Gordan problem, i.e.,
the decomposition of the tensor product of any two representations into indecomposable
summands, and the computation of the representation ring of
$\Rep^{lnlf} Q$, etc. Note that the tensor product given here is
different from the vertex-wise and arrow-wise tensor product used in
\cite{h1,h2,k1,k2}, which in general is not from bialgebra, and
therefore should provide different information for the categories of
quiver representations. This interesting problem is the third
motivation and will be treated in the future.

We remark that a standard dual process will give rise to natural
bialgebra structures on the path algebra and monoidal structures on
the category of representations of a finite quiver. We prefer the
path coalgebraic approach as this is more convenient for exposition
and, more importantly, allows infinite quivers.

Throughout, we work over a field $\k.$ Vector spaces,
algebras, coalgebras, bialgebras, linear mappings, and unadorned
$\otimes$ are over $\k.$ The readers are referred to \cite{mont1,sw}
for general knowledge of coalgebras and bialgebras, and to
\cite{ass,s,s2} for that of quivers and their applications to
(co)algebras and representation theory.

\section{Quivers, Representations and Path Coalgebras}

As preparation, in this section we recall some basic notions and
facts about quivers, representations and path coalgebras.

%\subsection{}
A quiver is a directed graph. More precisely, a quiver is a
quadruple $Q=(Q_0,Q_1,s,t),$ where $Q_0$ is the set of vertices,
$Q_1$ is the set of arrows, and $s,t: Q_1 \longrightarrow Q_0$ are
two maps assigning respectively the source and the target for each
arrow. Note that, in this paper the sets $Q_0$ and $Q_1$ are allowed
to be infinite. If $Q_0$ and $Q_1$ are finite, then we say $Q$ is a
finite quiver. For $a \in Q_1,$ we write $a:s(a) \To t(a).$ A vertex
is, by convention, said to be a trivial path of length $0.$ We also
write $s(g)=g=t(g)$ for each $g \in Q_0.$ The length of an arrow is
set to be $1.$ In general, a non-trivial path of length $n \ (\ge
1)$ is a sequence of concatenated arrows of the form $p=a_n \cdots
a_1$ with $s(a_{i+1})=t(a_i)$ for $i=1, \cdots, n-1.$ By $Q_n$ we
denote the set of the paths of length $n.$ A quiver is said to be
acyclic if it has no cyclic paths, i.e. nontrivial paths with
identical starting and ending vertices.

%\subsection{}
Let $Q$ be a quiver and $\k Q$ the associated path space which is
the $\k$-span of its paths. There is a natural coalgebra structure
on $\k Q$ with comultiplication as split of paths. Namely, for a
trivial path $g,$ set $\D(g)=g \otimes g$ and $\e(g)=1;$ for a
non-trivial path $p=a_n \cdots a_1,$ set
$$\D(p)=t(a_n) \otimes p + \sum_{i=1}^{n-1}a_n \cdots a_{i+1}
\otimes a_i \cdots a_1 +p \otimes s(a_1)$$ and $\e(p)=0.$ This is
the so-called path coalgebra of the quiver $Q.$

There exists on $\k Q$ an intuitive length gradation $\k
Q=\bigoplus_{n \geqslant 0}\k Q_n,$ and with which the
comultiplication $\D$ complies perfectly. It is clear that the path
coalgebra $\k Q$ is pointed, and the set of group-like elements
$G(\k Q)$ is $Q_0.$  Moreover, the coradical filtration of $\k Q$ is
$$ \k Q_0 \subseteq \k Q_0 \oplus \k Q_1 \subseteq \k Q_0 \oplus \k Q_1
\oplus \k Q_2 \subseteq \cdots \ ,$$ therefore it is coradically
graded in the sense of Chin-Musson \cite{cm}.

%\subsection{}
The path coalgebra is exactly the dual notion of the path algebra of
a quiver, which is certainly more familiar. Dual to the freeness of
path algebras, path coalgebras are cofree. Precisely, for an
arbitrary quiver $Q,$ the vector space $\k Q_0$ is a subcoalgebra of
$\k Q,$ and over the vector space $\k Q_1$ there is an induced $\k
Q_0$-bicomodule structure via
$$\d_L(a)=t(a) \otimes a, \quad \d_R(a)=a \otimes s(a)$$ for each $a
\in Q_1;$ the path coalgebra has another presentation as the
so-called cotensor coalgebra (see, e.g. \cite{oz,w})
$$\operatorname{CoT}_{\k Q_0}(\k Q_1)=\k Q_0 \oplus \k Q_1 \oplus \k
Q_1 \square \k Q_1 \oplus \cdots $$ and hence enjoys the following

\texttt{Universal Mapping Property}: Let $f: C \longrightarrow \k Q$
be a coalgebra map, $\pi_n: \k Q \longrightarrow \k Q_n$ the
canonical projection and set $f_n:=\pi_n \circ f: \ C
\longrightarrow \k Q_n$ for each $n \ge 0.$ Then $f_0$ is a
coalgebra map; $f_1$ is a $\k Q_0$-bicomodule map, where the $\k
Q_0$-bicomodule structure on $C$ is induced by $f_0;$ and for each
$n \ge 2,$ $f_n$ can be written as $f_1^{\otimes n} \circ
\D_C^{(n-1)},$ where $\D_C^{(n-1)}$ is the $(n-1)$-iterated action
of the comultiplication of $C.$ On the contrary, given a coalgebra
map $f_0 :\ C \longrightarrow \k Q_0$ and a $\k Q_0$-bicomodule map
$f_1:\ C \longrightarrow \k Q_1,$ set $f_n=f_1^{\otimes n} \circ
\D_C^{(n-1)}$ for each $n \ge 2,$ then as long as $f:=\sum_{n \ge
0}f_n$ is well-defined, it is the unique coalgebra map $f:\ C
\longrightarrow \k Q$ such that $f_0=\pi_0 \circ f$ and $f_1=\pi_1
\circ f.$

%\subsection{}
Let $Q$ be a quiver. A representation of $Q$ is a collection
$$V=(V_g,V_a)_{g \in Q_0, a \in Q_1}$$ consisting of a vector space
$V_g$ for each vertex $g$ and a linear map $V_a: V_{s(a)} \To
V_{t(a)}$ for each arrow $a.$ A morphism of representations $\phi: V
\To W$ is a collection $\phi=(\phi_g)_{g \in Q_0}$ of linear maps
$\phi_g:V_g \To W_g$ for each vertex $g$ such that
$W_a\phi_{s(a)}=\phi_{t(a)}V_a$ for each arrow $a.$ The category of
representations of $Q$ is denoted by $\Rep(Q).$ Given a
representation $V$ of $Q$ and a path $p,$ we define $V_p$ as
follows. If $p$ is trivial, say $p=g \in Q_0,$ then put
$V_p=\Id_{V_g}.$ For a non-trivial path $p=a_n \cdots a_2a_1,$ put
$V_p=V_{a_n} \cdots V_{a_2}V_{a_1}.$ A representation $V$ of $Q$ is
said to be locally nilpotent if for all $g \in Q_0$ and all $x \in
V_g,$ there exist at most finitely many paths $p$ with source $g$ satisfying $V_p(x) \neq
0.$ A representation is called locally finite if it is a directed union of finite-dimensional representations. We denote by $\Rep^{lnlf}(Q)$ the full subcategory of locally
nilpotent and locally finite representations of $\Rep(Q).$ It is well-known that the
category of right $\k Q$-comodules is equivalent to $\Rep^{lnlf}(Q),$
see \cite{ks}.

\section{Quiver Bialgebras}

In this section we show that the path coalgebra of an arbitrary
quiver can be endowed with natural bialgebra structures. A Gabriel
type theorem for pointed bialgebras is also given. Some examples are
presented.

%\subsection{}%bialgebra bimodules to graded bialgebras
We start with the definition of bialgebra bimodules. Let $B$ be a
bialgebra. A $B$-bialgebra bimodule is a vector space $M$ which is a
$B$-bimodule and simultaneously a $B$-bicomodule such that the
$B$-bicomodule structure maps are $B$-bimodule maps, or
equivalently, the $B$-bimodule structure maps are $B$-bicomodule
maps.

\begin{lemma}
Let $Q$ be a quiver. The associated path coalgebra $\k Q$ admits a
bialgebra structure if and only if $Q_0$ has a monoid structure and
$\k Q_1$ can be given a $\k Q_0$-bialgebra bimodule structure.
Moreover, the set of graded bialgebra structures on the path
coalgebra $\k Q$ is in one-to-one correspondence with the set of
pairs $(S,M)$ in which $S$ is a monoid structure on $Q_0$ and $M$ is
a $\k S$-bialgebra bimodule structure on $\k Q_1.$
\end{lemma}

\begin{proof}
Assume first that the path coalgebra $\k Q$ admits a bialgebra
structure. By considering its graded version induced by the
coradical filtration (see, e.g. \cite{mont1,n}), we can assume
further that the bialgebra structure on $\k Q$ is coradically
graded. Note that the identity $1$ is a group-like, so $1$ lies in
$Q_0$ which is the set of group-like elements of $\k Q.$ For any two
elements $g,h \in Q_0,$ we have $\D(gh)=\D(g)\D(h)=gh \otimes gh$
and $\e(gh)=\e(g)\e(h)=1,$ therefore $gh \in Q_0.$ Hence the
restriction of the multiplication of $\k Q$ onto $Q_0$ gives rise to
a monoid structure. The $\k Q_0$-bicomodule structure on $\k Q_1$ is
given as in Subsection 2.3. The multiplication of $\k Q_0$ provides
a bimodule structure on $\k Q_1.$ Finally note that the axioms for
bialgebras guarantee that the so-defined $\k Q_1$ is a $\k
Q_0$-bialgebra bimodule.

Conversely, assume that $Q_0$ can be endowed a monoid structure and
the vector space $\k Q_1$ has a $\k Q_0$-bialgebra bimodule
structure. By the method of Nichols \cite{n}, these data can be used
to construct a graded bialgebra structure on the path coalgebra $\k
Q$ by the universal mapping property. Nichols' construction was
applied to quiver setting for Hopf algebras in \cite{cr1, cr2,
g,gs}. For completeness we include the construction in below.

The cotensor coalgebra $\CoT_{\k Q_0} (\k Q_1)$ is exactly the path
coalgebra $\k Q.$ The $\k Q_0$-bimodule structure on $\k Q_1$ can be
extended to a multiplication on $\k Q$ via the universal mapping
property of $\k Q.$ Let $M_0: \ \k Q \otimes \k Q \longrightarrow \k
Q_0$ be the composition of the canonical projection $\pi_0 \otimes
\pi_0:\ \k Q \otimes \k Q \longrightarrow \k Q_0 \otimes \k Q_0$ and
the multiplication of the monoid algebra $\k Q_0,$ and $M_1: \k Q
\otimes \k Q \longrightarrow \k Q_1$ the composition of the
canonical projection
\[ \pi_0 \otimes \pi_1 \bigoplus \pi_1 \otimes
\pi_0: \ \k Q \otimes \k Q \longrightarrow \k Q_0 \otimes \k Q_1
\bigoplus \k Q_1 \otimes \k Q_0 \] and the sum of left and right
module actions. Then it is clear that $M_0$ is a coalgebra map and
$M_1$ is a $\k Q_0$-bicomodule map. Let $M_n=\D_2^{(n-1)} \circ
M_1^{\otimes n},$ where $\D_2$ is the coproduct of the tensor
product coalgebra $\k Q \otimes \k Q.$ For any path $p$ of length
$n,$ it is easy to see that $M_l(p)=0$ if $l \ne n.$ Therefore
$M=\sum_{n \ge 0}M_n$ is a well-defined coalgebra map and moreover
respects the length gradation. The associativity for the map $M$ can
be deduced from the associativity of the bimodule action without
difficulty by a standard application of the universal mapping
property as before. The unit map is obvious. Hence we have defined
an associative algebra structure and we obtain a graded bialgebra
structure on $\k Q.$

The latter one-to-one correspondence is obvious.
\end{proof}

%\subsection{}%construction of $\k Q_0$-bialgebra bimodule
Now we state our first main result.
\begin{theorem}
Let $Q$ be a quiver. The associated path coalgebra $\k Q$ always
admits bialgebra structures.
\end{theorem}

\begin{proof}
By Lemma 3.1, it is enough to provide a monoid structure on $Q_0$
and a $\k Q_0$-bialgebra bimodule on $\k Q_1.$ In the first place,
we fix the $\k Q_0$-bicomodule structure on $\k Q_1$ as in
Subsection 2.3. If $Q_0$ has only one element, then let it be the
unit group and take the trivial $\k Q_0$-bimodule structure on $\k
Q_1.$ Obviously, this defines a necessary bialgebra bimodule. Now we
assume $Q_0$ contains at least 2 elements. Take an arbitrary
element, say $e \in Q_0,$ and set it to be the identity, i.e., let
$eg=g=ge$ for all $g \in Q_0.$ Take any element $z \in Q_0$
other than $e,$ and make it to be a ``zero" element. That is, let
$gz=z=zg$ for any $g \in Q_0.$ For any two elements $g,h \in
Q_0-\{e,z\},$ set $gh=z.$ Here, $g=h$ is allowed. One can verify
without difficulty that this endows $Q_0$ with a monoid structure.
For the $\k Q_0$-bimodule structure on $\k Q_1,$ define $$e.a=a=a.e,
\quad f.a=0=a.f$$ for all $a \in Q_1$ and all $f \in
Q_0-\{e\}.$ Clearly, the bicomodule structure maps are bimodule maps
and hence we have obtained a $\k Q_0$-bialgebra bimodule structure
on $\k Q_1.$
\end{proof}

%\subsection{}
Note that, in the proof of the previous theorem we provide only a
very ``trivial" example of graded bialgebra structures for the path
coalgebra. The classification of all the graded bialgebra structures
on $\k Q,$ or equivalently the classification of all the suitable
monoid structures on $Q_0$ and $\k Q_0$-bialgebra bimodule
structures on $\k Q_1,$ is still not clear and is a very interesting
problem of quiver combinatorics.

However, if $Q_0$ can be afforded a group structure and further $\k
Q_1$ can be afforded a corresponding bialgebra bimodule structure
over $\k Q_0,$ then the situation is much clearer. Indeed, in this
case $\k Q_0$ becomes a Hopf algebra and $\k Q_1$ becomes a $\k
Q_0$-Hopf bimodule \cite{n}. Now the the fundamental theorem of Hopf
modules of Sweedler \cite[Theorem 4.1.1]{sw} can be applied, and the
category of $\k Q_0$-Hopf bimodules is proved to be equivalent to
the direct product of the representation categories of a class of
subgroups of $Q_0$ by Cibils and Rosso \cite{cr1}. Note that, a
quiver $Q$ with $Q_0$ having a group structure and $\k Q_1$ having a
$\k Q_0$-Hopf bimodule structure is far from arbitrary. Such quivers
are called covering quivers in \cite{gs} and Hopf quivers in
\cite{cr2}. It would be of interest to generalize the classification
of Hopf bimodules over groups to that of bialgebra bimodules over
monoids.

%\subsection{}%multiplication formula
For later use, we record the multiplication formula of paths, given
by Rosso's quantum shuffle product \cite{rosso}, for quiver
bialgebras. Given a quiver $Q,$ take a suitable monoid structure on
$Q_0$ and a $\k Q_0$-bialgebra bimodule structure on $\k Q_1.$ Let
$p$ be a path of length $l.$ An $n$-thin split of it is a sequence
$(p_1, \ \cdots, \ p_n)$ of vertices and arrows such that the
concatenation $p_n \cdots p_1$ is exactly $p.$ These $n$-thin splits
of $p$ are in one-to-one correspondence with the $n$-sequences of
$(n-l)$ 0's and $l$ 1's. Denote the set of such sequences by
$D_l^n.$ Clearly $|D_l^n|={n \choose l}.$ For $d=(d_1, \ \cdots, \
d_n) \in D_l^n,$ the corresponding $n$-thin split is written as
$dp=((dp)_1, \ \cdots, \ (dp)_n),$ in which $(dp)_i$ is a vertex if
$d_i=0$ and an arrow if $d_i=1.$

Let $\alpha=a_m \cdots a_1$ and $\beta=b_n \cdots b_1$ be a pair of
paths of length $m$ and $n$ respectively. Let $d \in D_m^{m+n}$ and
$\bar{d} \in D_n^{m+n}$ the complement sequence which is obtained
from $d$ by replacing each 0 by 1 and each 1 by 0. Define an element
in $\k Q_1^{\square_{m+n-1}},$ or equivalently in $\k Q_{m+n},$
$$(\alpha \cdot \beta)_d=[(d\alpha)_{m+n}.(\bar{d}\beta)_{m+n}] \cdots
[(d\alpha)_1.(\bar{d}\beta)_1],$$ where
$[(d\alpha)_i.(\bar{d}\beta)_i]$ is understood as the action of $\k
Q_0$-bimodule on $\k Q_1$ and these terms in different brackets are
put together by cotensor product, or equivalently concatenation. In
these notations, the formula of the product of $\alpha$ and $\beta$
is given as follows:
\begin{equation}
\alpha \cdot \beta=\sum_{d \in D_m^{m+n}}(\alpha \cdot \beta)_d \ .
\notag
\end{equation}

%\subsection{}%Gabriel type theorem for pointed bialgebras
Now we give a Gabriel type theorem for general pointed bialgebras,
which is in the same vein as the results of Chin and Montgomery
\cite{cm} for pointed coalgebras and of Van Oystaeyen and Zhang
\cite{oz} for pointed Hopf algebras. Let $B$ be a pointed bialgebra
and $\{B_i\}_{i \ge 0}$ be its coradical filtration. Let $\gr B =
B_0 \bigoplus B_1/B_0 \bigoplus B_2/B_1 \bigoplus \cdots \cdots$ be
the associated coradically graded bialgebra induced by the coradical
filtration.

\begin{proposition}
Let $B$ be a pointed bialgebra and $\gr B$ be its coradically graded
version. There exist a unique quiver $Q$ and a unique graded
bialgebra structure on $\k Q$ such that $\gr B$ can be realized as a
sub-bialgebra of $\k Q$ with $\k Q_0 \oplus \k Q_1 \subseteq \gr B.$
\end{proposition}

\begin{proof}
Note that $\gr B$ is still pointed and its coradical is $B_0.$ Let
$Q_0$ be the set of group-like elements of $\gr B.$ Then clearly
$Q_0$ is a monoid and $B_0=\k Q_0$ as bialgebras, where the latter
is the usual monoid bialgebra. Induced from the graded bialgebra
structure, $B_1/B_0$ is a $B_0=\k Q_0$-bialgebra bimodule. The
bicomodule structure maps are denoted as $\d_L$ and $\d_R.$ As a $\k
Q_0$-bicomodule, $B_1/B_0$ is in fact a $Q_0$-bigraded space.
Namely, write $M=B_1/B_0,$ then
$$M=\bigoplus_{g,h \in Q_0}\ ^gM^h,$$ where $^gM^h=\{m \in M |
\d_L(m)=g \otimes m, \ \d_R(m)=m \otimes h \}.$ We attach to these
data a quiver $Q$ as follows. Let the set of vertices be $Q_0.$ For
all $g,h \in Q_0,$ let the number arrows with source $h$ and
target $g$ be equal to the $\k$-dimension of the isotypic space
$^gM^h.$ Now we have obtained a quiver $Q$ and moreover $\k Q_1$ has
a $\k Q_0$-bialgebra bimodule structure which is identical to the
$B_0$-bialgebra bimodule $B_1/B_0.$ The $\k Q_0$-bialgebra bimodule
$\k Q_1$ gives rise to a unique graded bialgebra structure on $\k
Q.$ By the universal mapping property, the coalgebra map $\gr B
\stackrel{\pi_0}{\To} B_0 \simeq \k Q_0$ and the $\k Q_0$-bicomodule
map $\gr B \stackrel{\pi_1}{\To} B_1/B_0 \simeq \k Q_1$ determine a
unique coalgebra map $\Theta: \gr B \To \k Q.$ Here, by $\pi_i$ we
denote the canonical projection $\gr B \To B_i/B_{i-1}.$ By a
theorem of Heyneman and Radford, see e.g. \cite[Theorem
5.3.1]{mont1}, the coalgebra map $\Theta$ is injective since its
restriction to the first term of the coradical filtration is
injective. Again, by the universal mapping property, one can show
that the map $\Theta$ is also an algebra map. Therefore, $\Theta$ is
actually an embedding of bialgebras. The last condition $\k Q_0
\oplus \k Q_1 \subseteq \gr B$ guarantees that the quiver $Q$ is
unique.
\end{proof}

%\subsection{}%Examples of bialgebras on path coalgebras
In the following we give some examples of bialgebras on the path
coalgebras of quivers.

\begin{example} \rm
Let $\K_n$ be the $n$-Kronecker quiver, i.e. a quiver of the form
$$\xy (0,0)*{\bullet}; {\ar (2,2)*{}; (28,2)*{}}; {\ar (2,-2)*{}; (28,-2)*{}};
(15,0)*{.}; (15,1)*{.}; (15,-1)*{.}; (30,0)*{\bullet} \endxy$$
Denote the arrows as $a_1, \cdots, a_n.$ Let the source vertex be
$e$ and the target vertex be $z$ as in Theorem 3.2. Then there is a
bimodule structure on the space spanned by $\{a_i\}_{i=1}^n$ over
the monoid $\{ e, z\}$ defined by $$ e.a_i=a_i=a_i.e, \quad
z.a_i=0=a_i.z$$ for all $i.$ We have the following multiplication
formulas for the quiver bialgebra $\k \K_n:$ $$ a_i \cdot a_j=0,
\quad \forall \ 1 \le i,j \le n.$$
\end{example}

\begin{example}\rm
Let $\S_n$ be the $n$-subspace quiver, i.e. 
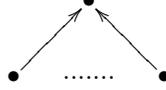
$$\xy (0,0)*{\bullet};
(20,0)*{\bullet}; (10,10)*{\bullet}; {\ar (1,1)*{}; (9,9)*{}}; {\ar
(19,1)*{}; (11,9)*{}}; (9,0)*{.}; (10,0)*{.}; (11,0)*{.}; (7,0)*{.};
(8,0)*{.}; (12,0)*{.}; (13,0)*{.}; \endxy $$ Denote the target
vertex by $e,$ the source vertices by $f_1, \cdots, f_n,$  and the
corresponding arrows by $a_1, \cdots, a_n.$ Assign $e$ to be the
identity, $f_1$ to be the ``zero" element, we get a monoid structure
on the set of vertices as in Theorem 3.2. The bimodule structure is
defined similarly $$ e.a_i=a_i=a_i.e, \quad f_i.a_j=0=a_j.f_i$$ for
all $1 \le i,j \le n.$ The multiplication of the quiver bialgebra
$\k \S_n$ is similar: $a_i \cdot a_j=0, \ \forall \ i,j.$
\end{example}

\begin{example}\rm
Let $\A_\infty$ be the quiver $$\xy (0,0)*{\bullet};
(10,0)*{\bullet}; (20,0)*{\bullet}; (30,0)*{\bullet};
(40,0)*{\bullet}; (47,0)*{\cdots\cdots}; {\ar (1,0)*{}; (9,0)*{}};
{\ar (11,0)*{}; (19,0)*{}}; {\ar (21,0)*{}; (29,0)*{}}; {\ar
(31,0)*{}; (39,0)*{}};
\endxy$$ Index the vertices $\g_i,$ from left to right, by the set of natural numbers
$\mathbb{N}=\{0,1,2,\cdots\cdots\}$ and consider the additive monoid
structure, i.e. $g_ig_j=g_{i+j}.$ Denote the arrow $g_i \To g_{i+1}$
by $a_i.$ Define a bialgebra bimodule structure on the space of
arrows by $$g_i.a_j=a_{i+j}, \quad a_j.g_i=q^ia_{i+j}$$ for all $i,j
\in \mathbb{N},$ where $q \in \k-\{0\}$ is a parameter. Assume
further that $q$ is not a root of unity. By a routine verification
one can show that the axioms of bialgebra bimodules are satisfied.
Let $p_i^l$ denote the path $a_{i+l-1} \cdots a_{i+1}a_i$ if $l \ge
1,$ or $g_i$ if $l=0.$ Apparently, $\{p_i^l\}_{i,l \ge 0}$ is a
basis of $\k \A_\infty.$ Then using Subsection 3.4 as well as
induction we have the following multiplication formula
$$ p_i^l \cdot p_j^m = q^{jl}{l+m \choose m}_q p_{i+j}^{l+m}$$ for
all $i,j,l,m \in \mathbb{N}.$ Here we use the quantum binomial
coefficients as in \cite{kassel}. For completeness, we recall their definition. For
any $q \in \k,$ integers $l,m \ge 0,$ define
\[
l_q=1+ q + \cdots +q^{l-1}, \ \ l!_q=1_q \cdots l_q, \ \ \mathrm{and} \ \
\binom{l+m}{l}_q=\frac{(l+m)!_q}{l!_qm!_q}.
\]
Clearly the quiver bialgebra $\k
\A_\infty$ is generated as an algebra by $g_1$ and $a_0$ with
relation $a_0g_1=qg_1a_0.$ Therefore, $\k \A_\infty$ is exactly the
quantum plane of Manin \cite{manin} and the bialgebra structure
given here is identical to that in \cite[P118]{kassel}.
\end{example}

We remark that the quivers in the previous examples admit no Hopf
algebra structures, as they are certainly not Hopf quivers. For
simple quivers as in Examples 3.4 and 3.5, it is not difficult to
classify all the possible graded bialgebra structures on the path
coalgebras. The bialgebra structure given in Example 3.6 is
different from the ``trivial" one as given in the proof of Theorem
3.2.

%\subsection{}%Examples of bialgebras on path subcoalgebras
Though theoretically any graded pointed bialgebra can be obtained as
a large sub-bialgebra of a quiver bialgebra, there seems no chance
to present a general method for such construction. However, in the
following we will show that a large class of pointed bialgebras
whose coradical filtration has length 2 can be constructed
systematically.

\begin{proposition}
Let $Q$ be a quiver. Assume that in $Q_0$ there is a sink (i.e.,
admitting only incoming arrows) or a source (i.e., admitting only
outgoing arrows). Then there exists on $\k Q$ a bialgebra structure
such that $\k Q_0 \oplus \k Q_1$ becomes its sub-bialgebra.
\end{proposition}

\begin{proof}
By assumption, there is a sink or a source in $Q_0.$ Now fix such a
vertex to be the identity $e$ of the monoid on $Q_0$ and take the
graded bialgebra structure on $\k Q$ as given in the proof of
Theorem 3.2. Let $B=\k Q_0 \oplus \k Q_1.$ We claim that $B$ is a
sub-bialgebra of $\k Q.$ In fact, we only need to verify that the
multiplication of arrows is closed in $B.$ Given an arbitrary pair
of arrows $a: g \To h$ and $b: u \To v,$ by Subsection 3.4 the
multiplication is $$a \cdot b = [a.v][g.b] + [h.b][a.u].$$ Clearly
the term $[a.v][g.b]$ survives, namely not =0, only if $v=g=e,$ but
this contradicts with the assumption that $e$ is a sink or a source.
Similarly we have $[h.b][a.u]=0.$ This shows that the multiplication
of $\k Q$ is indeed closed in $B.$
\end{proof}

\begin{remark}
Coalgebras with coradical filtration having length 2 are studied by Kosakowska and Simson in \cite{ks}, where a reduction to hereditary coalgebras is presented and the Gabriel quiver is discussed in term of irreducible morphisms. The dual of such a coalgebra is
an algebra of radical square zero. The class
of radical square zero algebras is very important in the
representation theory of Artin algebras (see e.g. \cite{ass, ars}). The
dual of the previous proposition asserts that every elementary
radical square zero algebra with condition $\Ext^1(S,-)=0,$ or
$\Ext^1(-,S)=0$ for some simple module $S$ has a bialgebra
structure, therefore its module category has natural tensor product.
\end{remark}

\section{Monoidal Structures over Quiver Representations}

In this section we consider the natural monoidal structures on the
categories of locally nilpotent representations of quivers arising
from bialgebra structures.

%\subsection{}%pointed monoidal categories
Recall that a monoidal categories is a sextuple
$(\C,\otimes,\1,\alpha,\lambda,\rho),$ where $\C$ is a category,
$\otimes: \C \times \C \to \C$ is a functor, $\1$ an object,
$\alpha: \otimes \circ ( \otimes \times \Id) \to \otimes \circ (\Id
\times \otimes), \lambda: \1 \otimes - \to \Id, \rho: - \otimes \1
\to \Id$ are natural isomorphisms such that the associativity and
unitarity constrains hold, or equivalently the pentagon and the
triangle diagrams are commutative, see e.g. \cite{maclane} for
detail.

Natural examples of monoidal structures are the categories of
$B$-modules and $B$-comodules with $B$ a bialgebra, see e.g.
\cite{kassel, mont1}. Recall that, if $U$ and $V$ are right
$B$-comodules and $U \otimes V$ the usual tensor product of
$\k$-spaces, then the comodule structure of $U \otimes V$ is given
by $u \otimes v \mapsto u_0 \otimes v_0 \otimes u_1v_1,$ where we
use the Sweedler notation $u \mapsto u_0 \otimes u_1$ for comodule
structure maps. The unit object is the trivial comodule $\k$ with
comodule structure map $k \mapsto k \otimes 1.$ On the other hand,
by the reconstruction formalism, monoidal categories with fiber
functors are coming in this manner, see e.g. \cite{egno, majid}.

The monoidal categories arising from quiver bialgebras (as their
right comodule categories) share a common property, i.e., their
simple objects all have $\k$-dimension 1 and consist of a monoid.
Stimulated by this and the notion of pointed tensor categories
introduced in \cite{egno}, we call a $\k$-linear monoidal category
pointed if the iso-classes of simple objects constitute a monoid
(under tensor product).

\emph{From now on, the field $\k$ is assumed to be algebraically
closed and the monoidal categories under consideration are
$\k$-linear abelian}. A monoidal category is said to be finite, if
the underlying category is equivalent to the category of
finite-dimensional comodules over a finite-dimensional coalgebra.
This is a natural generalization of the notion of finite tensor
categories of Etingof and Ostrik \cite{eo}.

Classification of finite monoidal categories is a fundamental
problem. Our results in Section 3 indicate that even finite pointed
monoidal categories are ``over" pervasive, it is necessary to impose
proper condition before considering the classification problem. In
the following, by taking advantage of the theory of quivers and
their representations, we give some classification results for some
classes of finite pointed monoidal categories.

An abelian category is said to be hereditary, if the extension
bifunctor $\Ext^n$ vanishes at each degree $n \ge 2.$ Next we
consider hereditary finite pointed monoidal categories. Though the
following results are direct consequences from some well-known
theorems of quivers and representations, we feel it is of interest
to include them here.

\begin{proposition}
A hereditary finite pointed monoidal category with a fiber functor
is equivalent to $(\Rep Q, \F)$ with $Q$ a finite acyclic quiver and
$\F$ the forgetful functor from $\Rep Q$ to the category $\Vec_\k$
of vector spaces.
\end{proposition}

\begin{proof}
By the standard reconstruction process, see e.g. \cite{egno,majid},
a finite monoidal category with fiber functor is equivalent to
$(B-\comod, \F)$ in which $B-\comod$ is the category of
finite-dimensional right comodules over a finite-dimensional
bialgebra $B$ and $\F$ is the forgetful functor. Note that, the
category of right comodules over a finite-dimensional bialgebra $B$
is pointed if and only if $B$ is pointed. Now by the Gabriel type
theorem for pointed coalgebras \cite{cm}, there exists a unique
quiver $Q$ such that $B$ is isomorphic to a large subcoalgebra, i.e.
includes the space spanned by the set of vertices and arrows, of the
path coalgebra $\k Q.$ The hereditary condition of the category
$B-\comod$ forces $B$ to be hereditary as a coalgebra. This
indicates that $B$ is isomorphic to $\k Q$ as a coalgebra. Since $B
\cong \k Q$ is finite-dimensional, the quiver $Q$ must be finite and
acyclic. Obviously, any representation of $Q$ is automatically
locally nilpotent, hence $B-\comod$ is equivalent to $\Rep Q.$
\end{proof}

Now we can use Gabriel's famous classification theorem \cite{gab} on
quivers of finite-representation type, i.e. admitting only finitely
many indecomposable representations up to isomorphism, to describe
hereditary pointed monoidal categories in which there are only
finitely many iso-classes of indecomposable objects. Follow the
terminology of \cite{qha3}, a finite monoidal category is said to be
of finite type if it has only finitely many iso-classes of
indecomposable objects.

\begin{corollary}
A hereditary pointed monoidal category of finite type with a fiber
functor is of the form $\Rep Q$ where $Q$ is a finite disjoint union
of quivers of $\rm ADE$ type.
\end{corollary}

%\subsection{}
Finally we give two examples of quiver monoidal categories. We also work
out their Clebsch-Gordan formula and representation ring
respectively.

\begin{example} \rm
Let $\A_n$ be the quiver $$\xy (0,0)*{\bullet}; (10,0)*{\bullet};
(20,0)*{\bullet}; (32,0)*{\bullet}; (42,0)*{\bullet};
(26,0)*{\cdots\cdots}; {\ar (1,0)*{}; (9,0)*{}}; {\ar (11,0)*{};
(19,0)*{}}; {\ar (33,0)*{}; (41,0)*{}};
\endxy$$ with $n \ge 2$ vertices $v_1, \cdots, v_n$ and $n-1$ arrows $a_1,\cdots,
a_{n-1}$ where $a_i:v_i \To v_{i+1}.$ Set $v_1$ to be the identity,
$v_2$ to be the zero element, take the monoid structure on $\{v_i|1
\le i \le n\}$ as in Theorem 3.2 and consider the corresponding
bialgebra structure with multiplication given by $$v_1 \cdot a_i
=a_i= a_i \cdot v_1, \ v_j \cdot a_i =0= a_i \cdot v_j \ (j \ge 2), \
a_i \cdot a_j=0 .$$ For a pair of integers $i,j$ satisfying $1 \le i
\le j \le n,$ define a representation $V(i,j)$ of $\A_n$ by
$$V(i,j)_{v_k}=\left\{
                \begin{array}{ll}
                  \k, & \hbox{$i \le k \le j$} \\
                  0, & \hbox{otherwise}
                \end{array}
              \right. \quad \rm{and} \quad V(i,j)_{a_k}=\left\{
                                           \begin{array}{ll}
                                             1, & \hbox{$i \le k \le j-1$} \\
                                             0, & \hbox{otherwise}
                                           \end{array}
                                         \right..$$
It is well-known that the set $\{ V(i,j) \ | \ 1 \le i \le j \le n
\}$ is a complete list of indecomposable representations of $\A_n.$
For the Clebsch-Gordan problem of the category $\Rep \A_n,$ it is
enough to consider the decomposition rule of $V(i,j) \otimes V(k,l)$
thanks to the Krull-Schmidt theorem. Given a representation
$V(i,j),$ let $e_s (i \le s \le j)$ denote a basis element of the
vector space $\k$ assigned to the $s$-th vertex. Recall that the
associated comodule structure map for $V(i,j)$ is given by
$$\delta(e_s)=\sum_{x=s}^j e_x \otimes p_s^{x-s},$$ where by $p_x^y$ we denote the path of length $y$ starting at the vertex $x.$
Now for $e_s \otimes e_t \in V(i,j) \otimes V(k,l),$ we have
$$\delta(e_s \otimes e_t)= \left\{
                               \begin{array}{ll}
                                \sum_{x=s}^j e_x \otimes e_t \otimes p_s^{x-s}+\sum_{y=t}^l e_s \otimes e_y \otimes p_t^{y-t} , & \hbox{$s=t=1$;} \\
                                \sum_{y=t}^l e_s \otimes e_y \otimes p_t^{y-t}, & \hbox{$s=1,t \ge 2$;} \\
                                \sum_{x=s}^j e_x \otimes e_t \otimes p_s^{x-s}, & \hbox{$s \ge 2,t=1$;} \\
                                e_s \otimes e_t \otimes v_2 , & \hbox{$s \ge 2, t \ge 2$.}
                               \end{array}
                             \right.
.$$ Therefore, it is clear that
$$ V(i,j) \otimes V(k,l) =\left\{
                            \begin{array}{ll}
                            V(i,j) \oplus V(k,l) \oplus V(2,2)^{(j-i)(l-k)+1}, & \hbox{$i=k=1$;} \\
                            V(k,l) \oplus V(2,2)^{(j-i)(l-k+1)}, & \hbox{$i=1,k \ge 2$;} \\
                            V(i,j) \oplus V(2,2)^{(j-i+1)(l-k)}, & \hbox{$i \ge 2, k=1$;} \\
                            V(2,2)^{(j-i+1)(l-k+1)}, & \hbox{$i \ge 2, k \ge 2$.}
                            \end{array}
                                                          \right.$$

\end{example}

\begin{example} \rm
Consider the infinite quiver $\A_\infty.$ Take the bialgebra
structure on $\k \A_\infty$ as in Example 3.6 and keep the notations
therein. For any pair of integers $i,j$ with $0 \le i \le j,$ define
a representation $V(i,j)$ of $\A_\infty$ as the previous example.
Clearly, $\{ V(i,j) \ | \ 0 \le i \le j \}$ is a complete set of
locally nilpotent and locally finite indecomposable representations of $\A_\infty.$ As
in Example 4.3, take a basis element $e_s$ for the vector space $\k$
in $V(i,j)$ attached to the $s$-th vertex $g_s.$ The corresponding
comodule structure map of $V(i,j)$ is $$ \delta(e_s) = \sum_{x=s}^j
e_x \otimes p_s^{x-s}.$$ Consider the tensor product $V(i,j) \otimes
V(k,l).$ For $e_s \otimes e_t \in V(i,j) \otimes V(k,l),$ we have
$$\delta(e_s \otimes e_t) = \sum_{x=s}^j \sum_{y=t}^l {x-s+y-t \choose y-t}_q e_x \otimes
e_y \otimes p_{s+t}^{x-s+y-t}. $$ With this, it is not hard to see
that
$$V(0,1) \otimes V(0,n) = V(0,n+1) \oplus V(1,n)$$ and
$$V(i,j) \otimes V(1,1)= V(i+1,j+1) = V(1,1) \otimes V(i,j).$$ This implies that the representation ring
of $\Rep^{lnlf} \A_\infty$ is generated by $V(0,1)$ and $V(1,1)$ and
is isomorphic to the polynomial ring in two variables $\Z[X,Y].$
\end{example}

\vskip 0.15cm

\noindent{\bf Acknowledgements:} The research of H.-L. Huang was partially
supported by SDNSF (grant no. 2009ZRA01128) and IIFSDU (grant no.
2010TS021). The paper was written while Huang was visiting the
Universidad de Almer\'ia and he acknowledges its
hospitality. Both authors are very grateful to the referees for their useful suggestions.


\begin{thebibliography}{99}

\bibitem{ass}
I. Assem, D. Simson, A. Skowronski, Elements of the Representation
Theory of Associative Algebras. Vol. 1. Techniques of Representation
Theory. London Mathematical Society Student Texts, 65. Cambridge
University Press, Cambridge, 2006.

\bibitem{ars}
M. Auslander, I. Reiten, S.O. Smal\o, Representation Theory of Artin
Algebras, Cambridge Studies in Adv. Math. 36, Cambridge Univ. Press,
1995.

\bibitem{chz}
X. Chen, H.-L. Huang, P. Zhang, Dual Gabriel theorem with applications, Sci. China Ser. A 49 (2006), no. 1, 9-26.

\bibitem{cz}
X. Chen, P. Zhang, Comodules of $U_q(sl_2)$ and modules of $SL_q(2)$ via quiver methods, J. Pure Appl. Algebra 211 (2007), no. 3, 862-876.

%\bibitem{ckq}
%W. Chin, M. Kleiner, D. Quinn, Almost split sequences for comodules,
%J. Algebra 249 (2002) 1-19.

\bibitem{cm}
W. Chin, S. Montgomery, Basic coalgebras, Modular interfaces
(Riverside, CA, 1995), 41-47, AMS/IP Stud. Adv. Math. 4, Amer. Math.
Soc., Providence, RI, 1997.

\bibitem{cmu}
W. Chin, I.M. Musson, The coradical filtration for quantized
enveloping algebras, J. London Math. soc. (2) 53 (1996) 50-62.

%\bibitem{c1}
%C. Cibils, A quiver quantum group, Comm. Math. Phys. 157 (1993)
%459-477.

%\bibitem{c2}
%C. Cibils, Half-quantum groups at roots of unity, path algebras, and
%representation type, Internat. Math. Res. Notices 12 (1997) 541-553.

\bibitem{cr1}
C. Cibils, M. Rosso, Alg\`{e}bres des chemins quantiques, Adv. Math.
125 (1997) 171-199.

\bibitem{cr2}
C. Cibils, M. Rosso, Hopf quivers, J. Algebra 254 (2002) 241-251.

\bibitem{egno}
P. Etingof, S. Gelaki, D. Nikshych, V. Ostrik, Tensor categories,
lecture note for the MIT course 18.769, 2009. available at:
www-math.mit.edu/~etingof/tenscat.pdf

\bibitem{eo}
P. Etingof, V. Ostrik, Finite tensor categories, Moscow Math. J. 4
(2004) 627-654.

\bibitem{gab}
P. Gabriel, Unzerlegbare Darstellungen I, Manuscripta
Math. 6 (1972) 71-103.

\bibitem{g}
E.L. Green, Constructing quantum groups and Hopf algebras
from coverings, J. Algebra 176 (1995) 12-33.

%\bibitem{green}
%J.A. Green, Locally finite representations, J.
%Algebra 41 (1976) 137-171.

\bibitem{gs}
E.L. Green, \O. Solberg, Basic Hopf algebras and
quantum groups, Math. Z. 229 (1998) 45-76.

\bibitem{h1}
M. Herschend, Tensor products on quiver representations, J. Pure
Appl. Algebra 212 (2008) 452-469.

\bibitem{h2}
M. Herschend, On the representation ring of a quiver, Algebr.
Represent. Theory, 12 (2009) 513-541.

\bibitem{qha3}
H.-L. Huang, G. Liu, Y. Ye, Quviers, quasi-quantum groups and finite
tensor categories, Comm. Math. Phys. 303 (2011) 595-612.

\bibitem{kassel}
C. Kassel, Quantum Group, Graduate Texts in Math. 155,
Springer-Verlag, New York, 1995.

\bibitem{k1}
R. Kinser, The rank of a quiver representation, J. Algebra 320(6)
(2008) 2363-2387.

\bibitem{k2}
R. Kinser, Rank functions on rooted tree quivers, Duke Math. J.
152(1) (2010) 27-92.

\bibitem{ks}
J. Kosakowska, D. Simson, Bipartite coalgebras and a reduction functor for coradical square complete coalgebras, Colloq. Math. 112 (2008), no. 1, 89-129.

%\bibitem{li}
%F. Li, Weak Hopf algebras and some new solutions of the quantum
%Yang-Baxter equation, J. Algebra 208(1) (1998) 72-100.

\bibitem{maclane}
S. Mac Lane, Categories for the working mathematicians, 3rd Edition,
Graduate Texts in Math. 5, Springer-Verlag, New York, 1998.

\bibitem{majid}
S. Majid, Foundations of Quantum Group Theory, Cambridge University
Press, Cambridge, 1995.

%\bibitem{majid2}
%S. Majid, Doubles of quasitriangular Hopf algebras, Comm. Algebra
%19(11) (1991), 3061-3073.

\bibitem{manin}
Yu.I. Manin, Some remarks on Koszul algebras and quantum groups,
Ann. Inst. Fourier. (Grenoble) 37(4) (1987) 191-205.

\bibitem{mont1}
S. Montgomery, Hopf Algebras and Their Actions
on Rings, CBMS Regional Conf. Series in Math. 82, Amer. Math. Soc.,
Providence, RI, 1993.

\bibitem{mont2}
S. Montgomery, Indecomposable coalgebras, simple comodules and
pointed Hopf algebras, Proc. of the Amer. Math. Soc. 123 (1995)
2343-2351.

\bibitem{n}
W.D. Nichols, Bialgebras of type one, Communications in
Algebra 6(15) (1978) 1521-1552.

%\bibitem{rt}
%D.E. Radford, J. Towber, Yetter-Drinfeld categories associated to an
%arbitrary bialgebra, J. Pure Appl. Algebra 87 (1993) 259-279.

\bibitem{rosso}
M. Rosso, Quantum groups and quantum shuffles, Invent. Math. 133
(1998) 399-416.

\bibitem{s}
D. Simson, Coalgebras, comodules, psseudocompact algebras and tame comodule type, Colloq. Math. 90 (2001) 101-150.

\bibitem{s2}
D. Simson, Coalgebras of tame comodule type, comodule categories, and a tame-wild dichotomy problem, in: Representations of algebras and related topics, 561-660, EMS Ser. Congr. Rep., Eur. Math. Soc., Z\"urich, 2011.

%\bibitem{sch}
%P. Schauenburg, Hopf modules and Yetter-Drinfeld modules, J. Algebra
%169(3) (1994) 874-890.

\bibitem{sw}
M. Sweedler, Hopf algebras, W. A. Benjamin, Inc., New
York, 1969.

%\bibitem{taft}E.J. Taft, The order of the antipode of finite
%dimensional Hopf algebras, Prc. Nat. Acad. Sci. USA 68 (1971)
%2631-2633.

%\bibitem{tak}
%M. Takeuchi, Free Hopf algebras generated by
%coalgebras, J. Math. Soc. Japan 23 (1971) 561-582.

\bibitem{oz}
F. Van Oystaeyen, P. Zhang, Quiver Hopf algebras, J. Algebra 280(2)
(2004) 577-589.

\bibitem{w}
D. Woodcock, Some categorical remarks on the representation theory
of coalgebras, Comm. Algebra 25 (1997) 2775-2794.

\end{thebibliography}
\end{document}